\documentclass[a4paper, 11pt]{amsart}

\usepackage[T1]{fontenc}
\usepackage[utf8]{inputenc}

\usepackage{amssymb}
\usepackage{mathrsfs}
\usepackage{hyperref}

\newcommand{\ov}[1]{\overline{#1}}

\allowdisplaybreaks[2]

\newtheorem{theorem}{Theorem}[section]

\newtheorem{lemma}[theorem]{Lemma}
\newtheorem{corollary}[theorem]{Corollary}
\theoremstyle{definition}

\newtheorem{definition}[theorem]{Definition}

\numberwithin{equation}{section}

\newcommand{\N}{\mathbb N}

\newcommand{\AAA}[1]{\mathscr{A}_{#1}}
\newcommand{\AAAi}[1]{\widetilde{\mathscr{A}}_{#1}}

\newcommand{\aaa}{\textbf{a}}

\newcommand{\Om}{\Omega}
\newcommand{\om}{\omega}
\newcommand{\omz}{\underline{\omega}}
\newcommand{\Omi}{\widetilde{\Omega}}
\newcommand{\Omii}{\Omega'}

\newcommand{\aaai}{\widetilde{\aaa}}
\newcommand{\aaaii}{\aaa'}

\newcommand{\Yo}{U}

\newcommand{\yo}{u}

\newcommand{\pcp}[3]{#1 \langle #2 \rangle \amalg #1 [#3]}

\newcommand{\slo}[1]{{\rm sl}_{#1}(F)}

\newcommand{\lama}[3]{\lambda_{#1, #3}^{#2}}
\newcommand{\lamt}[2]{\lambda_{#1}^{#2}}

\newcommand{\Ev}[2]{{\rm Ev}_{#1 ; #2}}

\newcommand{\lngth}[1]{|#1|}
\newcommand{\ins}[3]{#1 \triangleleft_{#2} #3}

\newcommand{\minfty}[1]{M_\infty(#1)}
\newcommand{\sinfty}[1]{{\rm sl}_\infty(#1)}

\DeclareMathOperator{\diag}{diag}
\DeclareMathOperator{\tr}{tr}

\title[Images of mult. poly.]{Images of multilinear polynomials in the algebra of finitary matrices contain  trace zero matrices}

\author{Daniel Vitas}
\address{Department of Mathematics, Faculty of Mathematics and Physics,  University of Ljubljana, Slovenia}
\email{daniel.vitas@student.fmf.uni-lj.si}

\thanks{\emph{Mathematics Subject Classification} (2020). 16R99,  16W25}
\keywords{Multilinear polynomial, finitary matrices,  L'vov-Kaplansky conjecture.}

\begin{document}

\begin{abstract} Let $F$ be an infinite field and let $f$ be a nonzero multilinear  polynomial with coefficients in $F$. We prove that for every positive integer $d$ there exists
a positive integer  $s$ such that $f(M_{s}(F))$, the image of $f$ in  $M_{s}(F)$, contains all trace zero $d\times d$ matrices. In particular, the image of $f$ in the algebra of all finitary matrices contains all trace zero finitary matrices.
\end{abstract}

\maketitle

\section{Introduction}

Let $F$ be a field. By $F \langle X_1,\ldots,X_n \rangle$ we denote the free algebra in $X_i$ over $F$; its elements are called \emph{noncommutative polynomials}. The \emph{image} of a noncommutative polynomial $f \in F \langle X_1,\ldots,X_n \rangle$ in the $F$-algebra $A$ is the set
\[
	f(A) = \{f(a_1,\ldots,a_n) \mid a_1,\ldots,a_n \in A\} \text{.}
\]

We say that $f$ is \emph{multilinear} if it is of the form
\[
	f = \sum_{\sigma \in S_n} \lambda_{\sigma} X_{\sigma(1)} \ldots X_{\sigma(n)}
\]
for some $\lambda_\sigma \in F$.
The L'vov-Kaplansky conjecture states that the image of a multilinear polynomial on the matrix algebra $M_d(F)$ is a vector space---in fact, since the linear span of the image of a polynomial is a Lie ideal, it can only be one of the four vector spaces: $\{0\}$, the space of scalar matrices $F$, the space of trace zero matrices $\slo{d}$, or the whole algebra $M_d(F)$. In \cite{KBMR}, Kanel-Belov, Malev, and Rowen proved this conjecture for  $2\times 2$ matrices (over a quadratically closed field $F$). Although  extensively studied by several mathematicians, the conjecture is at present unsolved even  in the $3 \times 3$ case. We refer the reader to \cite{RYMKB} for a detailed survey of the results regarding the L'vov-Kaplansky conjecture and its variations. 

In \cite{Vi}, the  author proved that if  $A$ is an algebra having a surjective inner derivation, then $f(A)=A$ for every nonzero multilinear polynomial. Since an example of such an algebra $A$ is 
End$_F(V)$, the algebra of endomorphisms of an infinite-dimensional vector space $V$, this result can be viewed as the solution of an infinite-dimensional version of the L'vov-Kaplansky conjecture.
In the present paper, we will use some ideas from \cite{Vi} to establish the following theorem, which still involves an algebra that is larger than one would wish, but nevertheless brings us closer to the classical L'vov-Kaplansky conjecture.

\begin{theorem} \label{lep izrek}
	Let $F$ be an infinite field and let $d \in \N$. For every nonzero multilinear polynomial $f$ there exists an $s \in \N$ such that $\slo{d} \subseteq f(M_{s}(F))$. 
\end{theorem}

Here, an embedding of a smaller matrix algebra into a larger one should be understood as
\[
M_d(F) = \begin{pmatrix} M_d(F) & 0 \\ 0 & 0 \end{pmatrix} \subseteq M_s(F) \text{.}
\]
This theorem is also related to Mesyan's \cite{Me} weaker form of the L'vov-Kaplansky conjecture, which states that  
$\slo{d} \subseteq f(M_d(F))$ provided that $d$ is larger than $n$, the degree of $f$ (for $n=4$ this was proved in \cite{BW}).

Let $\minfty{F}$ denote the algebra of all {\em finitary matrices}, i.e., (countably) infinite matrices with only finitely many nonzero entries. Denoting by $\sinfty{F}$ the space 
of all trace zero finitary matrices, we thus have the following corollary to Theorem \ref{lep izrek}.

\begin{corollary} \label{posledica} 	Let $F$ be an infinite field and let $f$ be a nonzero multilinear polynomial. Then 
 $ \sinfty{F}  \subseteq f(\minfty{F})\text{.}$
\end{corollary}

If the polynomial $f$ is a sum of commutators, the reverse inclusion also holds.
If not, then it is easy to see that   $f(\minfty{F})$ contains a matrix $a$ with nonzero trace (see, e.g., \cite[Theorem 4.5]{BK}).
 Since $x-\frac{\tr(x)}{\tr(a)} a \in \sinfty{F}$ for every $x \in \minfty{F}$, it follows from Corollary \ref{posledica} that 
every matrix in $ \minfty{F}$ is a sum of two matrices from 
$f(\minfty{F})$. This result is in the spirit of  \cite{B}.
We conjecture, however, that it  is not optimal and that  $f(\minfty{F})$ is actually equal to the whole $\minfty{F}$ 
(and hence, in particular, $f(\minfty{F})$ is a vector space for every multilinear polynomial $f$). We leave this as an open problem.

At first glance, the paper is very similar to \cite{Vi}, with some parts being almost identical. However, the problem studied here is more subtle, so nontrivial modifications had to be made.
It should be  remarked that the general setting, introduced in Section \ref{s2}, is indeed similar, but not equal to that in \cite{Vi}. Section \ref{s3} is devoted to proving
Theorem \ref{mt} which is slightly more general than Theorem \ref{lep izrek}.

\section{Admissible partially commutative polynomials} \label{s2}

Let $F$ be a field, let $n \in \N$, and let $\Omega = \{ \omega_1,\ldots, \omega_m\} \subseteq \N \setminus \{1,\ldots,n\}$ be a finite set (the reason for such notation is that we will change $\Omega$ in the course of proof). By
\[
	\pcp{F}{X_1,\ldots,X_n}{U_\omega \mid \omega \in \Omega}
\]
we denote the coproduct (see, e.g., \cite[Section 1.4]{Rings}) of the free algebra $F \langle X_1,\ldots,X_n\rangle$ and the algebra of commutative polynomials $F[U_\omega \mid \omega \in \Omega]$; we call its elements \emph{partially commutative polynomials}. We may think of them as polynomials in variables $X_i, U_\omega$, where $U_\omega$ commute among themselves, but do not commute with $X_i$.

Let $A$ be a unital algebra over $F$, let $x_1,\ldots,x_n \in A$, and let $u_{\omega_1},\ldots,u_{\omega_m} \in A$ be elements that commute among themselves. By
\[
\Ev{x_1, \ldots, x_n}{u_{\omega_1},\ldots,u_{\omega_m}}: \pcp{F}{X_1,\ldots,X_n}{U_\omega \mid \omega \in \Omega} \rightarrow A
\]
denote the algebra homomorphism sending $X_i$ to $x_i$ and $U_\omega$ to $u_\omega$.
Since the above map extends the standard evaluation homomorphisms of $F \langle X_1,\ldots,X_n\rangle$ and $F[U_\omega \mid \omega \in \Omega]$, respectively, its existence follows from the universal property of the coproduct.

For a partially commutative polynomial $f \in \pcp{F}{X_1,\ldots,X_n}{U_\omega \mid \omega \in \Omega}$, define its \emph{image} in the algebra $A$ as
\[
	f(A) = \left\{\Ev{x_1,\ldots,x_n}{\yo_{\om_1},\ldots,\yo_{\om_m}}(f) \mid x_i, \yo_{\om_j} \in A,~ \yo_{\omega_j} \yo_{\omega_k} = \yo_{\omega_k} \yo_{\omega_j} ~\text{for all $j,k$} \right\} \text{.}
	\]
Note that in the case where $f$ is a noncommutative polynomial, i.e., an element of the subalgebra $F \langle X_1,\ldots,X_n\rangle$ of $\pcp{F}{X_1,\ldots,X_n}{U_\omega \mid \omega \in \Omega}$, this notion of the image of $f$ coincides with the standard one.

We will say that the  sequences $\aaa^i = (\aaa_1^i,\ldots,\aaa_{k_i}^i)$, $i=1,\ldots,n$, containing the elements from $\Om$  form a \emph{partition} of $\Om$ if they are strictly increasing (i.e., $\aaa_1^i < \ldots < \aaa_{k_i}^i$) and, for each $\om \in \Om$, there exist uniquely determined $i$ and $j$ such that $\aaa^i_j = \om$. By $\lngth{\aaa^i}$ we denote the length of the sequence, i.e., $k_i$.
We define $\AAA{\Om}$ as the set of all $n$-tuples of sequences
$\aaa = (\aaa^1,\ldots,\aaa^n)$
such that $\aaa^1,\ldots,\aaa^n$ form a partition of $\Om$.

As usual, we write 
\[
[x_1,x_2] = x_1 x_2 - x_2 x_1
\]
for the commutator of the elements $x_1$ and $x_2$. More generally, we write
\[
[x_1, x_2, \ldots, x_n] = [x_1, [x_2, \ldots, x_n]] \text{.}
\]
For any $\aaa \in \AAA{\Om}$, we define
\[
X^\aaa_i = [\Yo_{\aaa_1^i}, \ldots, \Yo_{\aaa_{k_i}^i}, X_i] 
\]
if $k_i>0$ and $X_i^\aaa = X_i$ if $k_i = 0$.
We extend this definition by setting
\begin{equation*}
	\left( X_{i_1} X_{i_2} \ldots X_{i_k}\right)^{\aaa} = X_{i_1}^{\aaa} X_{i_2}^{\aaa} \ldots  X_{i_k}^{\aaa} \text{,}
\end{equation*}
for all $i_1,\ldots,i_k \in \{1,\ldots, n\}$.

We generalize the notion of a multilinear polynomial  as follows.

\begin{definition}
	A partially commutative polynomial
	\[
	f \in \pcp{F}{X_1,\ldots,X_n}{\Yo_\om \mid \om \in \Om} 
	\]
	is {\em admissible}
	if it is of the form 
	\begin{equation*}
		f = \sum_{\sigma \in S_n} \sum_{\aaa \in \AAA{\Omega}} \lamt{\sigma}{\aaa} \left( X_{\sigma(1)} X_{\sigma(2)} \ldots X_{\sigma(n)} \right)^{\aaa}
	\end{equation*}
	for some  $\lamt{\sigma}{\aaa}\in F$. 
\end{definition}

Multilinear noncommutative polynomials are exactly admissible partially commutative polynomials for $\Omega = \emptyset$. 
To give a different example, let $n=2$ and $\Omega = \{3,4\}$. Then
$$\AAA{\Om} = \{ ((3,4), \emptyset),~((3), (4)),~((4),(3)),~(\emptyset,(3,4)) \} \text{,}$$
and admissible polynomials in  $ \pcp{F}{X_1,X_2}{\Yo_3,\Yo_4} $ are of the form
\begin{align*}
	f &=  \lamt{{\rm id}}{((3,4), \emptyset)} [\Yo_3, \Yo_4, X_{1}] X_{2} + \lamt{(12)}{((3,4), \emptyset)} X_{2} [\Yo_3, \Yo_4, X_{1}] \\
	&+  \lamt{{\rm id}}{((3), (4))} [\Yo_3, X_{1}] [\Yo_4, X_{2}] + \lamt{(12)}{((3), (4))} [\Yo_4, X_{2}] [\Yo_3, X_{1}] \\
	&+  \lamt{{\rm id}}{((4),(3))} [\Yo_4, X_{1}] [\Yo_3, X_{2}] + \lamt{(12)}{((4),(3))} [\Yo_3, X_{2}] [\Yo_4, X_{1}] \\
	&+  \lamt{{\rm id}}{(\emptyset,(3,4))} X_{1} [\Yo_3, \Yo_4, X_{2}] + \lamt{(12)}{(\emptyset,(3,4))} [\Yo_3, \Yo_4, X_{2}] X_{1}\text{.} \qedhere
\end{align*}

By definition, the vector space of admissible partially commutative polynomials is linearly spanned by $\left( X_{\sigma(1)}  \ldots X_{\sigma(n)} \right)^{\aaa}$. The following lemma states that these elements actually form its basis. Its proof is very similar to that of \cite[Proposition 2.4]{Vi}, so we omit it.

\begin{lemma}
	\label{lin neodvisnost}
	For any $n \in \N$ and a finite set $\Om \subseteq \N \setminus \{1,\ldots,n\}$,
	\[
	\left\{ \left( X_{\sigma(1)} X_{\sigma(2)} \ldots X_{\sigma(n)} \right)^{\aaa} \mid \sigma \in S_n, ~\aaa \in \AAA{\Om} \right\}
	\]
	is a linearly independent set.
\end{lemma}

\section{Main theorem}\label{s3}

Let $F$ be an infinite field and let $d \in \N$.
In this section we will prove Theorem \ref{mt}, which states that for any nonzero admissible partially commutative polynomial $f$ there exists an $s \in \N$ such that $\slo{d} \subseteq f(M_{s}(F))$.

The proof of the theorem is by induction on the number of noncommuting variables $X_1,\ldots,X_n$. Before considering the base case, we prove a lemma which slightly extends the well-known fact 
that,  
 in characteristic $0$,  every trace zero matrix is similar to a \emph{hollow matrix}, i.e., a matrix having only zeros on the diagonal (see, e.g., \cite[Proposition 1.8]{Am}). This is no longer true
if $F$ has prime characteristic. Indeed, the identity matrix can have trace zero, but obviously is not similar to a hollow matrix. However, we can redeem this by increasing the size of matrices.

\begin{lemma} \label{zabavna lema}
	For any $d \in \N$ and every matrix $a \in \slo{d}$ there exists an invertible matrix $p \in M_{d+1}(F)$ such that $pap^{-1} \in M_{d+1}(F)$ is a hollow matrix.
\end{lemma}

\begin{proof}
	We proceed by induction on $d$. The lemma is obviously true for $d=1$, so assume that $d > 1$ and that the lemma is true for $d-1$. Let $a \in \slo{d}$ be a $d \times d$ matrix with trace $0$.
	
	First, consider the case where there exists a nonzero vector $v \in F^d$ that is not an eigenvector of $a$. Then we can extend the linearly independent set $\{v,av\}$ to a basis of the space $F^d$; let $q \in M_d(F)$ be the transition matrix (from the standard basis to the new one). Then
	\[
	qaq^{-1}= \begin{pmatrix} 0& x^\intercal \\ y & b\end{pmatrix}
	\]
	for some $x,y \in F^{d-1}$ and $b \in M_{d-1}(F)$. Since $\tr(b) = \tr(a) = 0$, by the induction hypothesis, there exists a matrix $r \in M_d(F)$ such that $rbr^{-1}$ is a hollow matrix. Thus, for
	\[
	p = \begin{pmatrix} 1 &  \\  & r \end{pmatrix} \begin{pmatrix} q &  \\  & 1 \end{pmatrix} \in M_{d+1}(F) \text{,}
	\]
	the matrix $pap^{-1} \in M_{d+1}(F)$ is hollow.
	
	Now, assume that all nonzero vectors are eigenvectors of $a$. This implies that $a$ is a scalar multiple of the identity matrix---without loss of generality we may assume $a = {\rm I}$. Since $0 = \tr(a) = d$, the characteristic of $F$ divides $d$. Let $e_1,\ldots,e_{d+1} \in F^{d+1}$ be the standard basis of $F^{d+1}$. The vectors $f_i = e_i + e_{d+1}$, $i=1,\ldots,d$, and $f_{d+1} = e_{d+1} - \sum_{j=1}^d e_j$ form a basis, since $$e_i = -\sum_{\substack{j=1 \\ j \neq i}}^{d+1} f_j \text{,}$$
	for $i=1,\ldots,d$, and $e_{d+1} = \sum_{j=1}^{d+1} f_j$. Let $p \in M_{d+1}(F)$ be the transition matrix (from the standard basis to the new one). Since $a e_i = e_i$ for $i=1,\ldots,d$ and $a e_{d+1} = 0$, we have
	\[
	a f_i = e_i = - \sum_{\substack{j=1 \\ j \neq i}}^{d+1} f_j \text{,}
	\]
	for $i=1,\ldots,d$, and
	\[
	a f_{d+1} = - \sum_{i=1}^d e_i = \sum_{i=1}^d \sum_{\substack{j=1 \\ j \neq i}}^{d+1} f_j = -\sum_{j=1}^d f_j \text{.}
	\]
	Therefore, the matrix $p a p^{-1} \in M_{d+1}(F)$ is hollow.
\end{proof}

We will now establish the basis of our induction. 
\begin{lemma} \label{baza}
	For a nonzero admissible partially commutative polynomial
	$$f \in \pcp{F}{X_1}{\Yo_\om \mid \om \in \Om} \text{,}$$
	we have
	$\slo{d} \subseteq f(M_{d+1}(F))$.
\end{lemma}

\begin{proof}
	Let $\Om = \{\om_1,\ldots,\om_m\}$ with $\om_1<\ldots<\om_m$. As $n=1$, we have
	$\AAA{\Om} = \{ (\omz)\}$ with $\omz = (\om_1,\ldots,\om_m)$.
	Thus,
	\[
	f = \lambda X_1^{(\omz)} = \lambda [\Yo_{\om_1}, \ldots, \Yo_{\om_m}, X_1]
	\]
	for some nonzero $\lambda \in F$. Let $a \in \slo{d}$ be a trace zero matrix. By Lemma \ref{zabavna lema}, there exists an invertible matrix $p \in M_{d+1}(F)$ such that $pap^{-1} \in M_{d+1}(F)$ is hollow. Since $F$ is infinite, we can take $\yo = \diag(\alpha_1,\ldots,\alpha_{d+1}) \in M_{d+1}(F)$ with pairwise distinct $\alpha_i \in F$. Since, for $y \in M_{d+1}(F)$, we have $[u,y]_{ij} = (\alpha_i - \alpha_j) y_{ij}$, there exists a matrix $x \in M_{d+1}(F)$ such that
	\[
	pap^{-1} = [\underbrace{\yo,\ldots,\yo}_m,x] \text{.}
	\]
Hence, the  image of $f$ contains
	\begin{align*}
		\Ev{\lambda^{-1} p^{-1}xp}{p^{-1} \yo p,\ldots,p^{-1} \yo p}(f) = \lambda p^{-1}[\underbrace{\yo,\ldots,\yo}_m, \lambda^{-1} x]p = a \text{.}
	\end{align*}
	This proves that  $\slo{d} \subseteq f(M_{d+1}(F))$.
\end{proof}

Before making the induction step, we will prove two lemmas. The first one is just an elementary observation involving the standard matrix units $e_{ij}$.

\begin{lemma} \label{ikomlema} Let $A$ be an arbitrary unital algebra, let
	$k \in \N$, and let $$v = \sum_{i=1}^{k} e_{i,i+1} + e_{k+1,1} \in M_{k+1}(A).$$ Then 
		\[
	[\underbrace{v,\ldots,v}_{j}, e_{k+1,1}] = \sum_{s=0}^j (-1)^{s} {j \choose s} e_{k+1-j+s,1+s} 
	\]
	for every $j=0,1,\ldots,k$.
	In particular, for $a \in A$, we have
	\[
	[\underbrace{v,\ldots,v}_{k}, a e_{k+1,1}] = \diag(a,*,\ldots,*) \text{.}
	\]
\end{lemma}

\begin{proof}
	We proceed by induction on $j$. The lemma is obviously true for $j=0$, so assume that $0<j \leq k$ and that the lemma is true for $j-1$. Using the induction hypothesis, we have
	\begin{align*}
		[\underbrace{v,\ldots,v}_{j}, e_{k+1,1}] &= [v, [\underbrace{v,\ldots,v}_{j-1}, e_{k+1,1}]]  \\
		&= \left[ v, \sum_{s=0}^{j-1} (-1)^{s} {j-1 \choose s} e_{k+1-(j-1)+s,1+s} \right] \\
		&= \sum_{s=0}^{j-1} (-1)^{s} {j-1 \choose s} [v,e_{k+2-j+s,1+s}] \text{.}
	\end{align*}
	Note that
	\begin{align*}
	v e_{k+2-j+s,1+s} &= e_{k+1-j+s,1+s}\\
	e_{k+2-j+s,1+s}v &= e_{k+2-j+s,2+s} \text{.}
	\end{align*}
	Hence,
	\begin{align*}
		[\underbrace{v,\ldots,v}_{j}, e_{k+1,1}]
		&= \sum_{s=0}^{j-1} (-1)^{s} {j-1 \choose s} \left( e_{k+1-j+s,1+s} - e_{k+2-j+s,2+s} \right) \\
		&= \sum_{s=0}^{j-1} (-1)^{s} {j-1 \choose s} e_{k+1-j+s,1+s} + \sum_{s=0}^{j-1} (-1)^{s+1} {j-1 \choose s} e_{k+2-j+s,2+s} \text{.}
	\end{align*}
	By changing the index of summation in the second sum, we see that
	\begin{align*}
		[\underbrace{v,\ldots,v}_{j}, e_{k+1,1}]
		&= \sum_{s=0}^{j-1} (-1)^{s} {j-1 \choose s} e_{k+1-j+s,1+s} + \sum_{s=1}^{j} (-1)^{s} {j-1 \choose s-1} e_{k+1-j+s,1+s} \\
		&= e_{k+1-j,1} + \sum_{s=1}^{j-1} (-1)^{s} \left({j-1 \choose s} +  {j-1 \choose s-1} \right) e_{k+1-j+s,1+s}\\ &+ (-1)^j e_{k+1,1+j} \text{.}
	\end{align*}
	Using
	\[
	{j-1 \choose s} + {j-1 \choose s-1} = {j \choose s}
	\]
	we obtain the conclusion of the lemma.
\end{proof}

The next lemma  will enable us to reduce the number of noncommutative variables in a suitable way.

\begin{lemma} \label{x v y}
	Let	$f \in \pcp{F}{X_1,\ldots,X_n}{\Yo_\om \mid \om \in \Om}$
	(with $n\geq 2$) be an admissible polynomial of the form
	\[
	f = \sum_{\sigma \in S_{n-1}} \sum_{j = 1}^n \sum_{\aaa \in \AAA{\Om}} \lama{\sigma}{\aaa}{j} \left( X_{\sigma(1)} \ldots X_{\sigma(j-1)} X_{n} X_{\sigma(j)} \ldots X_{\sigma(n-1)} \right)^{\aaa} \text{.}
	\]
	Let $k \in \N_0$ be such that for every $\aaa \in \AAA{\Om}$, $\lngth{\aaa^n} < k$ implies $\lama{\sigma}{\aaa}{j} = 0$ for every $\sigma \in S_{n-1}$ and every $j \in \{1, \ldots, n\}$. Let $\omz = (\om_1,\ldots,\om_k)$ be a part of some partition of $\Om$. Set $\Omi = \Om \setminus \{\om_1,\ldots,\om_k\}$ and let $A$ be an arbitrary unital algebra. Then  the partially commutative polynomial
	$$g \in \pcp{F}{X_1,\ldots,X_{n-1}}{\Yo_\om \mid \om \in \Omi \cup \{n\}}$$ 
	defined by
	\[
	g = \sum_{\sigma \in S_{n-1}} \sum_{j=1}^n \sum_{\aaai \in \AAAi{\Om}}\lama{\sigma}{\aaai}{j} \left( X_{\sigma(1)} \ldots X_{\sigma(j-1)} \right)^{\aaai} \Yo_n \left( X_{\sigma(j)}  \ldots X_{\sigma(n-1)} \right)^{\aaai} \text{,}
	\]
	where $\AAAi{\Om} = \{ \aaai \in \AAA{\Om} \mid \aaai^n = \omz \}$, satisfies
	$g(A)e_{11} \subseteq f(M_{k+1}(A))$.
\end{lemma}

\begin{proof}
	Let $\Om = \{\om_1,\ldots,\om_k,\om_{k+1},\ldots,\om_m\}$. Take $$x_1,\ldots,x_{n-1},\yo_n, \yo_{\om_{k+1}}, \ldots, \yo_{\om_{m}} \in A$$ such that $\yo_n, \yo_{\om_{k+1}}, \ldots, \yo_{\om_{m}}$ commute with each other.
	To prove the lemma we have to find $\ov{x}_1,\ldots,\ov{x}_n,\ov{u}_{\omega_1}, \ldots, \ov{u}_{\omega_m} \in M_{k+1}(A)$ such that
	\begin{equation} \label{zelena enakost}
	\Ev{\ov{x}_1,\ldots,\ov{x}_n}{\ov{\yo}_{\om_1},\ldots,\ov{\yo}_{\om_m}}(f) = \Ev{x_1,\ldots,x_{n-1}}{\yo_n,\yo_{\om_{k+1}},\ldots,\yo_{\om_m}}(g) e_{11}
	\end{equation}
	and $\ov{u}_{\omega_1}, \ldots, \ov{u}_{\omega_m}$ commute.
	
	Set
	\begin{align*}
		&\ov{x}_i = x_i e_{11} \,\, \text{for $i=1,\ldots,n-1$,} \\
		&\ov{x}_n = u_n e_{k+1,1} \text{,}\\
		&\ov{\yo}_{\om_j} = v \,\, \text{for $j=1,\ldots,k$,}\\
		&\ov{\yo}_{\om_l} =  \yo_{\om_l} {\rm I} \,\, \text{for $l=k+1,\ldots,m$,}
	\end{align*}
	where  $v$ is the matrix from Lemma \ref{ikomlema}.
	The matrices $\ov{\yo}_{\om_1},\ldots,\ov{\yo}_{\om_m}$ commute with each other, since so do the elements $\yo_{\om_{k+1}}, \ldots, \yo_{\om_{m}}$.
	We claim that  \eqref{zelena enakost} holds.
	Fix $\sigma \in S_{n-1}$ and $j \in \{1, \ldots, n\}$, and take an $\aaa \in \AAA{\Om}$. Consider the expression $$ \Ev{\ov{x}_1,\ldots,\ov{x}_n}{\ov{\yo}_{\om_1},\ldots,\ov{\yo}_{\om_m}}(\lama{\sigma}{\aaa}{j} X_n^{\aaa}) \text{.}$$
	If $\lngth{\aaa^n} < k$, then, by our assumption, $\lama{\sigma}{\aaa}{j}=0$ and this expression is zero. If the sequence $\aaa^n$ contains $\om_j$ for $j > k$, then
	\begin{align*}
		\Ev{\ov{x}_1,\ldots,\ov{x}_n}{\ov{\yo}_{\om_1},\ldots,\ov{\yo}_{\om_m}}(X_n^{\aaa}) = [\ov{\yo}_{\aaa^n_1},\ldots,\ov{\yo}_{\aaa^n_{\lngth{\aaa^n}}},\ov{x}_n] = 0 \text{,}
	\end{align*}
since $\ov{\yo}_{\om_j} = \yo_{\om_j} {\rm I}$ commutes with all $\ov{\yo}_{\om_s}$ and $\ov{x}_n$.
	Therefore, the above expression can be nonzero only if $\aaa^n$ contains at least $k$ elements from $\{\om_1,\ldots,\om_k\}$, i.e.,
	$\aaa^n = \omz$. In this case, by Lemma \ref{ikomlema},
	\[
	\Ev{\ov{x}_1,\ldots,\ov{x}_n}{\ov{\yo}_{\om_1},\ldots,\ov{\yo}_{\om_m}}(X_n^{\aaa}) = [\underbrace{v,\ldots,v}_k, u_n e_{k+1,1}] = \diag(\yo_n,*,\ldots,*) \text{.}
	\]
	For such an $\aaa$ and $i=1,\ldots,n-1$, we have 
	\begin{align*}
		\Ev{\ov{x}_1,\ldots,\ov{x}_n}{\ov{\yo}_{\om_1},\ldots,\ov{\yo}_{\om_m}}(X_i^{\aaa}) &= [\ov{\yo}_{\aaa^i_1},\ldots,\ov{\yo}_{\aaa^i_{\lngth{\aaa^i}}},\ov{x}_i] \\
		&= [\yo_{\aaa^i_1}{\rm I},\ldots,\yo_{\aaa^i_{\lngth{\aaa^i}}}{\rm I},x_i e_{11}] \\
		&= [\yo_{\aaa^i_1},\ldots,\yo_{\aaa^i_{\lngth{\aaa^i}}},x_i] e_{11} \text{.}
	\end{align*}
Consequently,
	\begin{align*}
		&\Ev{\ov{x}_1,\ldots,\ov{x}_n}{\ov{\yo}_{\om_1},\ldots,\ov{\yo}_{\om_m}}(f)\\
		=\,&\sum_{\sigma \in S_{n-1}} \sum_{j=1}^n \sum_{\aaa \in \AAA{\Om}} \Ev{\ov{x}_1,\ldots,\ov{x}_n}{\ov{\yo}_{\om_1},\ldots,\ov{\yo}_{\om_m}} \left( \left( X_{\sigma(1)} \ldots X_{\sigma(j-1)} \right)^{\aaa} \right)  \\ 
		&\cdot \Ev{\ov{x}_1,\ldots,\ov{x}_n}{\ov{\yo}_{\om_1},\ldots,\ov{\yo}_{\om_m}} \left( \lama{\sigma}{\aaa}{j}  X_n^{\aaa} \right)
		\Ev{\ov{x}_1,\ldots,\ov{x}_n}{\ov{\yo}_{\om_1},\ldots,\ov{\yo}_{\om_m}} \left( \left( X_{\sigma(j)} \ldots X_{\sigma(n-1)} \right)^{\aaa} \right) \\
		=\,& \sum_{\sigma \in S_{n-1}} \sum_{j=1}^n \sum_{\aaai \in \AAAi{\Omega}} \Ev{x_1,\ldots,x_{n-1}}{\yo_n,\yo_{\om_{k+1}},\ldots,\yo_{\om_m}} \left( \left( X_{\sigma(1)} \ldots X_{\sigma(j-1)} \right)^{\aaai} \right) e_{11}  \\ 
		&\cdot \lama{\sigma}{\aaai}{j} \diag(\yo_n,*,\ldots,*)
		\Ev{x_1,\ldots,x_{n-1}}{\yo_n,\yo_{\om_{k+1}},\ldots,\yo_{\om_m}} \left( \left( X_{\sigma(j)} \ldots X_{\sigma(n-1)} \right)^{\aaai} \right) e_{11} \\
		=\,& \Ev{x_1,\ldots,x_{n-1}}{\yo_n,\yo_{\om_{k+1}},\ldots,\yo_{\om_m}} (g) e_{11} \text{.} \qedhere
	\end{align*}
\end{proof}

We are now in a position to prove our main theorem.

\begin{theorem} \label{mt}
	Let $F$ be an infinite field and let $d \in \N$. For every nonzero admissible partially commutative polynomial $f$ there exists an $s \in \N$ such that $\slo{d} \subseteq f(M_{s}(F))$. 
\end{theorem}

\begin{proof}
	We proceed by induction on $n$, i.e., the number of noncommutative variables $X_1,\ldots,X_n$ involved in $f$. The case where $n=1$ was considered in Lemma \ref{baza}.
	
	Let $n > 1$ and assume the theorem is true for all nonzero admissible partially commutative polynomials in $n-1$ noncommuting variables. We can write $f \in \pcp{F}{X_1,\ldots,X_n}{\Yo_\om \mid \om \in \Om}$ as
	\[
	f = \sum_{\sigma \in S_{n-1}} \sum_{j = 1}^n \sum_{\aaa \in \AAA{\Om}} \lama{\sigma}{\aaa}{j} \left( X_{\sigma(1)} \ldots X_{\sigma(j-1)} X_{n} X_{\sigma(j)} \ldots X_{\sigma(n-1)} \right)^{\aaa} 
	\]
	for some $\lama{\sigma}{\aaa}{j} \in F$, not all zero.
	Suppose the theorem is not true, i.e., $$\slo{d} \not \subseteq f(M_s(F))$$ for all $s \in \N$.

	Let $k $  be the smallest nonnegative integer such that $\lama{\sigma}{\aaa}{j} \neq 0$ for some $\sigma \in S_{n-1}$, $j = 1, \ldots, n$, and $\aaa \in \AAA{\Om}$ with $\lngth{\aaa^n}=k$.
	Note that  $k$ satisfies the assumption of Lemma \ref{x v y}.
	Let $\aaa^n = \omz = (\om_1,\ldots,\om_k)$ be the $n$-th component of a partition $\aaa$ such that $\lama{\sigma}{\aaa}{j} \neq 0$ for some $\sigma \in S_{n-1}$ and $j = 1, \ldots, n$.
Our goal is to prove that for each $i = 1,\ldots, n$, 
	\begin{equation} \label{delna vsota}
		\sum_{j=1}^i \lama{\sigma}{\aaai}{j} = 0
	\end{equation}
	for every $\sigma \in S_{n-1}$ and every $\aaai \in \AAAi{\Om} = \{ \aaai \in \AAA{\Om} \mid \aaai^n = \omz \}$. This  obviously implies $\lama{\sigma}{\aaai}{j} = 0$ for every  $\sigma \in S_{n-1}$, every $\aaai \in \AAAi{\Om}$, and every $j = 1,\ldots, n$, which  contradicts our choice of $\omz$.

	Take $g \in \pcp{F}{X_1,\ldots,X_{n-1}}{\Yo_\om \mid \om \in \Omi \cup \{n\}}$ (recall that $\Omi = \Om \setminus \{\om_1,\ldots,\om_k\}$) defined by
	\[
	g = \sum_{\sigma \in S_{n-1}} \sum_{j=1}^n \sum_{\aaai \in \AAAi{\Om}}\lama{\sigma}{\aaai}{j} \left( X_{\sigma(1)} \ldots X_{\sigma(j-1)} \right)^{\aaai} \Yo_n \left( X_{\sigma(j)}  \ldots X_{\sigma(n-1)} \right)^{\aaai} \text{.}
	\]
	By Lemma \ref{x v y}, we have $g(A)e_{11} \subseteq f(M_{k+1}(A))$ for an arbitrary unital algebra $A$.
	Since $g$ does not involve the variable $X_n$, we can replace $\aaai$ by the $(n-1)$-tuple $\aaa$ obtained by taking the first $n-1$ components of $\aaai$. Such tuples $\aaa$ are exactly the elements of $\AAA{\Omi}$.
	Thus,
	\begin{align*}
		g = \sum_{\sigma \in S_{n-1}} \sum_{j=1}^n \sum_{\aaa \in \AAA{\Omi}}\lama{\sigma}{\aaai}{j} \left( X_{\sigma(1)} \ldots X_{\sigma(j-1)} \right)^{\aaa} \Yo_n \left( X_{\sigma(j)} \ldots X_{\sigma(n-1)} \right)^{\aaa} \text{,}
	\end{align*}
	where $\AAA{\Omi}$ contains $(n-1)$-tuples and $\aaai$ is the $n$-tuple obtained by adding the sequence $\omz$ to the end of $\aaa$.

	Let
	\[
	\pi:\pcp{F}{X_1,\ldots,X_{n-1}}{\Yo_\om \mid \om \in \Omi \cup \{n\}} \rightarrow \pcp{F}{X_1,\ldots,X_{n-1}}{\Yo_\om \mid \om \in \Omi}
	\]
	be the homomorphism that sends $\Yo_n$ to $1$ and fixes each $X_1,\ldots,X_{n-1}$ and the remaining $\Yo_\om$. We have
	\[
	\pi(g) = \sum_{\sigma \in S_{n-1}} \sum_{\aaa \in \AAA{\Omi}} \left( \sum_{j=1}^n \lama{\sigma}{\aaai}{j} \right) \left( X_{\sigma(1)} \ldots X_{\sigma(j-1)} X_{\sigma(j)}  \ldots X_{\sigma(n-1)} \right)^{\aaa} \text{.}
	\]
	Obviously, $\pi(g)(A) \subseteq g(A)$  holds for every unital algebra $A$. We claim that $\pi(g) = 0$. Indeed, if this was not true, then, 
	since $\pi(g)$ is an admissible partially commutative polynomial in $n-1$ noncommuting variables, it would follow from
	the induction hypothesis that there exists an $s \in \N$ such that
	\begin{align*}
	\slo{d} &\subseteq \pi(g)(M_s(F)) \subseteq g(M_s(F)) = g(M_s(F)) e_{11} \\ &\subseteq f(M_{k+1}(M_s(F))) = f(M_{(k+1)s}(F)) \text{,}
	\end{align*}
	which contradicts our initial assumption.
	Now, Lemma \ref{lin neodvisnost} implies  \eqref{delna vsota} for $i=n$.

	Using the just proven equality $\lama{\sigma}{\aaai}{n} = - \sum_{j=1}^{n-1} \lama{\sigma}{\aaai}{j}$, we have
	\begin{align*}
		g &= \sum_{\sigma \in S_{n-1}} \sum_{\aaa \in \AAA{\Omi}} \sum_{j=1}^{n-1} \lama{\sigma}{\aaai}{j} \left( X_{\sigma(1)} \ldots X_{\sigma(j-1)} \right)^{\aaa} \left[ \Yo_n , \left( X_{\sigma(j)}  \ldots X_{\sigma(n-1)} \right)^{\aaa} \right] \text{.}
	\end{align*}
	For $\aaa \in \AAA{\Omi}$, denote by $\ins{\aaa}{i}{n}$ the partition of $\Omii = \Omi \cup \{ n \}$ obtained by adding $n$ to the beginning of the sequence $\aaa^i$. The set $\AAA{\Omii}$ is in bijective correspondence with the disjoint union $\sqcup_{i=1}^{n-1} \AAA{\Omi}$ via $\aaaii = \ins{\aaa}{\sigma(i)}{n}$ (for some fixed permutation $\sigma$).
	By definition, we have
	\begin{align*}
		\left[ \Yo_n, X_{\sigma(i)}^{\aaa} \right]= X_{\sigma(i)}^{\ins{\aaa}{\sigma(i)}{n}} \text{,}
	\end{align*}
	and thus, by using the formula $[X, YZ] = [X, Y]Z + Y[X,Z]$ several times,
	\begin{align*}
		&\left[ \Yo_n , \left( X_{\sigma(j)}  \ldots X_{\sigma(n-1)} \right)^{\aaa} \right]\\
		=\,& \sum_{i = j}^{n-1} \left( X_{\sigma(j)}  \ldots X_{\sigma(i-1)} \right)^{\aaa} \left[ \Yo_n , X_{\sigma(i)}^{\aaa} \right] \left( X_{\sigma(i+1)}  \ldots X_{\sigma(n-1)} \right)^{\aaa} \\
		=\,& \sum_{i = j}^{n-1} \left( X_{\sigma(j)} \ldots X_{\sigma(n-1)} \right)^{\ins{\aaa}{\sigma(i)}{n}} \text{.}
	\end{align*}
	Therefore,
	\begin{align*}
		g&= \sum_{\sigma \in S_{n-1}} \sum_{\aaa \in \AAA{\Omi}} \sum_{j=1}^{n-1} \lama{\sigma}{\aaai}{j} \left( X_{\sigma(1)} \ldots X_{\sigma(j-1)} \right)^{\aaa} \sum_{i = j}^{n-1} \left( X_{\sigma(j)} \ldots X_{\sigma(n-1)} \right)^{\ins{\aaa}{\sigma(i)}{n}} \\
		&= \sum_{\sigma \in S_{n-1}} \sum_{\aaa \in \AAA{\Omi}} \sum_{j=1}^{n-1} \sum_{i = j}^{n-1} \lama{\sigma}{\aaai}{j} \left( X_{\sigma(1)} \ldots X_{\sigma(n-1)} \right)^{\ins{\aaa}{\sigma(i)}{n}}\text{.}
	\end{align*}
	By changing the order of summation and using the aforementioned bijective correspondence, we obtain
	\begin{align*}
		g &= \sum_{\sigma \in S_{n-1}} \sum_{i = 1}^{n-1} \sum_{\aaa \in \AAA{\Omi}} \left( \sum_{j=1}^{i} \lama{\sigma}{\aaai}{j} \right) \left( X_{\sigma(1)} \ldots X_{\sigma(n-1)} \right)^{\ins{\aaa}{\sigma(i)}{n}} \\
		 &= \sum_{\sigma \in S_{n-1}} \sum_{\aaaii \in \AAA{\Omii}} \left( \sum_{j=1}^{i} \lama{\sigma}{\aaai}{j} \right) \left( X_{\sigma(1)} \ldots X_{\sigma(n-1)} \right)^{\aaaii}.
	\end{align*}
	Here, $\sigma(i)$ is the uniquely determined component of $\aaaii$ which contains $n$, $\aaa$ is the $(n-1)$-tuple obtained by omitting  $n$ in the $\sigma(i)$th component of the sequence $\aaaii$, and $\aaai$ is the $n$-tuple obtained from $\aaa$ as before---$i$ and $\aaai$ thus only depend on $\sigma$ and $\aaaii$.
	We have $g=0$, since otherwise, by the induction hypothesis ($g$ is an admissible partially commutative polynomial in $n-1$ noncommuting variables), there would exist an $s \in \N$ such that
	\[
	\slo{d} \subseteq g(M_s(F)) \subseteq f(M_{(k+1)s}(F)) \text{.}
	\]
	Now, Lemma \ref{lin neodvisnost} implies  \eqref{delna vsota} for every $i = 1,\ldots, n-1$. As the $i=n$ case was established earlier, this concludes the proof.
\end{proof}

Since multilinear noncommutative polynomials are special
examples of admissible partially commutative polynomials, Theorem \ref{mt} directly implies Theorem \ref{lep izrek}.

\section*{Acknowledgment}

The author would like to thank his supervisor Matej Bre\v sar for his continued interest, generous encouragement, and invaluable assistance in improving this paper.


\begin{thebibliography}{99}
	
	\bibitem{Am} S.\ Amitsur, L.\ Rowen, Elements of reduced trace 0, {\em Israel J. Math.} \textbf{87} (1994) 161–179.
	
	\bibitem{Rings}K.\ I.\ Beidar, W.\ S.\ Martindale 3rd, A.\ V.\ Mikhalev, \emph{Rings with generalized identities}, Marcel Dekker, Inc., 1996.
	 
	
	\bibitem{B}M.\ Bre\v sar,  Commutators and images of noncommutative polynomials, {\em Adv. Math.} {\bf 374} (2020), 107346, 21 pp.
	
	\bibitem{BK} M.\ Bre\v sar, I.\ Klep, Values of noncommutative polynomials, Lie skew-ideals and tracial Nullstellensätze, {\em Math. Res. Lett.} {\bf 16} (2009), 605–626.
	
	\bibitem{BW} D.\ Buzinski, R.\ Winstanley, On multilinear polynomials in four variables
	evaluated on matrices, {\em Linear Algebra Appl.} \textbf{439} (2013),  2712–2719.
	
	\bibitem{KBMR}A.\ Kanel-Belov,  S.\ Malev,  L.\ Rowen, 
	The images of non-commutative polynomials evaluated on $2\times 2$ matrices,
	{\em Proc. Amer. Math. Soc.} {\bf 140} (2012),  465--478.
	
	\bibitem{RYMKB}A.\ Kanel-Belov, S.\ Malev, L.\ Rowen,  R.\ Yavich,
Evaluations of noncommutative polynomials on algebras: methods and problems, and the L'vov-Kaplansky conjecture, {\em SIGMA Symmetry Integrability Geom. Methods Appl.} {\bf  16}
 (2020), Paper No. 071, 61 pp.
 
	
	\bibitem{Me} Z.\ Mesyan, Polynomials of small degree evaluated on matrices, {\em Linear and Multilinear Algebra} \textbf{61} (2013), 1487-1495.
	

	\bibitem{Vi} D.\ Vitas, Multilinear polynomials are surjective on algebras with surjective inner derivations, {\em J.\ Algebra}  \textbf{565} (2021), 255-281.
	
\end{thebibliography}
\end{document}